\documentclass[12pt]{article}
\usepackage{amsmath,amssymb,amsthm,mathtools}
\usepackage{url}
\usepackage{xcolor}

\usepackage[hidelinks]{hyperref}
\hypersetup{pdftitle={Counterexamples to Hardy--Littlewood maximal inequalities on solvable groups}}

\newtheorem{maintheorem}{Theorem}

\newtheorem{theorem}{Theorem}[section]
\newtheorem{proposition}[theorem]{Proposition}
\newtheorem{lemma}[theorem]{Lemma}
\newtheorem{corollary}[theorem]{Corollary}
\theoremstyle{remark}

\numberwithin{equation}{section}

\newcommand{\LL}{\mathcal L}
\newcommand{\Sol}{\operatorname{Sol}}
\newcommand{\BS}{\operatorname{BS}}

\newcommand{\one}{\mathbf 1}
\newcommand{\R}{\mathbb R}
\newcommand{\Z}{\mathbb Z}
\newcommand{\Nzero}{\mathbb N_0}
\newcommand{\Ftwo}{\mathbb Z/2\mathbb Z}

\newcommand{\1}{\mathbf{1}}

\title{Counterexamples to Hardy--Littlewood maximal inequalities
on solvable groups}
\author{Koji Fujiwara
\thanks{supported by JSPS Grant-in-Aid for Scientific Research 25H00588}
}
\date{}

\begin{document}
\maketitle

\begin{abstract}
We prove that the centered Hardy--Littlewood maximal operator is not of
weak type $(p,p)$ for any $1\le p<\infty$ in three solvable settings:
the lamplighter group $(\Z/2\Z)\wr\Z$ with the
switch--walk--switch word metric, the Baumslag--Solitar groups
$\BS(1,n)$, $n\ge2$, with their  digit word metrics, and the
three-dimensional Sol group 
with any left-invariant Riemannian metric and its Riemannian volume.

To the best of our knowledge, we give the first examples of
finitely generated groups, equipped with word metrics and counting
measure, for which the centered Hardy--Littlewood maximal operator
fails to be of weak type $(1,1)$.

\end{abstract}



\section{Introduction}

\subsection{Background}
Let $(X,d,\mu)$ be a metric measure space in which every ball of positive
radius has positive finite measure. We write
$B_d(x,r)=\{y:d(x,y)\le r\}$ and consider the centered Hardy--Littlewood
maximal operator
\begin{equation}\label{eq:maximal-general}
 M_df(x)=\sup_{r>0}\frac{1}{\mu(B_d(x,r))} \int_{B_d(x,r)}|f(y)|\,d\mu(y).
\end{equation}

For $1\le p<\infty$, {\it weak type} $(p,p)$ means that a finite constant $C_p$
exists such that
$$
 \mu\{x:M_df(x)>\lambda\}
 \le C_p\lambda^{-p}\int_X|f|^p\,d\mu
 \qquad(f\in L^p(X,\mu),\ \lambda>0).
$$
When $\mu$ is counting measure, the integral is a sum and we write
$\ell^p$ in place of $L^p$.

Throughout the paper, $\Nzero=\{0,1,2,\ldots\}$.
For a group $G$ with a finite symmetric generating set $S$, let
\[
 |g|_S=\min\{k\in\Nzero:g=s_1\cdots s_k,\ s_i\in S\},
\]
be the word length of $g$. We define $|e|_S=0$.
 The associated left-invariant
word metric and  the ball of radius $r\ge0$  centered at $g$ are defined by
\[
 d_S(g,h)=|g^{-1}h|_S,\quad
 B_S(g,r)=\{h\in G:d_S(g,h)\le r\}.
\]
We have
$B_S(g,r)=gB_S(e,r)$ and $|B_S(g,r)|=|B_S(e,r)|$.
We write $M_S$ for the operator in \eqref{eq:maximal-general} with this
metric and counting measure. 

\bigskip

We say that $\mu$ is {\it doubling} if there is a constant $C$ such that
\[
\mu(B_d(x,2r))\le C\,\mu(B_d(x,r))
\quad (x\in X,\ r>0).
\]
The doubling condition gives weak type $(1,1)$ by a standard argument using the Vitali covering lemma. This includes Euclidean spaces and finitely generated groups of
polynomial growth with word metrics and counting measure. 

Among spaces of exponential growth, 
Str\"omberg proved weak type \((1,1)\) for the centered Hardy–Littlewood maximal operator on noncompact symmetric spaces, with their Riemannian metrics and Riemannian volume measures ~\cite{Stromberg}. Anker, Damek, and Yacoub established the
corresponding result on harmonic $AN$ groups~\cite[Corollary (3.22)]{AnkerDamekYacoub}.

Gaudry, Giulini, Hulanicki, and Mantero~\cite{Gaudry}
proved weak type $(1,1)$ for maximal operators defined using
certain families of rectangles on solvable Lie groups.
They asked whether a similar result
holds for Riemannian balls with respect to some left-invariant
Riemannian metric.

Corollary~\ref{cor:sol-all-riemannian} gives a negative answer
to this question in general: on $\mathrm{Sol}$, the centered
maximal operator is not of weak type $(1,1)$ for any
left-invariant Riemannian metric.

In the discrete setting, regular trees satisfy weak type \((1,1)\); this follows from work of Rochberg and Taibleson~\cite{RT}, and was later proved directly by Naor and Tao ~\cite[Theorem~1.5]{NaorTao}.
 More recently, Nevo and the author~\cite{FujiwaraNevo} proved weak type \((1,1)\) for every word metric on a non-elementary word-hyperbolic group. 

For broader background on Hardy–Littlewood maximal inequalities in metric measure spaces, see~\cite{NaorTao}.

Counterexamples are known among infinite graphs and Riemannian manifolds;
see, for example, \cite[Section 3.4]{SoriaTradacete} and \cite[Section 7]{MedaPigolaSettiVeronelli}, respectively.

By contrast, we will give counterexamples in homogeneous settings. We consider Cayley graphs with counting measure and connected Lie groups with left-invariant Riemannian metrics and Riemannian volume. In both cases, all balls of a given radius have the same measure.

\subsection{Main results}
The {\it lamplighter group} is the restricted wreath product
\[
\mathcal L=(\mathbb Z/2\mathbb Z)\wr \mathbb Z
=\left(\bigoplus_{n \in \mathbb Z}\mathbb Z/2\mathbb Z\right)\rtimes \mathbb Z,
\]
where \(\mathbb Z\) acts by shifting the coordinates. 

For $i\in\Z$, let $\delta_i\in \bigoplus_{n}\mathbb Z/2\mathbb Z$ be the indicator of $\{i\}$,
with values in $\Ftwo$. Set $a=(\delta_0,0)$ and $t=(0,1)$.
 One may interpret \(a\) as
switching the lamp at the current position of the lamplighter, while \(t\) moves the
lamplighter one step to the right.
We use the {\it switch–walk–switch} generating set
\[
S_{\mathrm{sws}}=\{t,at,ta,ata\}^{\pm1}.
\]

The lamplighter group is metabelian and amenable, but has exponential
growth. It is finitely generated, but 
not finitely presented.
See Section~\ref{sec:sws} for a detailed description of $\mathcal L$ and this generating
set.

\begin{maintheorem}[The lamplighter group]\label{thm:lamplighter}
Let $\mathcal L$ be equipped with counting measure and with the
switch--walk--switch word metric associated with $S_{\mathrm{sws}}$.
Then the centered Hardy--Littlewood maximal operator is not of weak
type $(p,p)$ for any $1\le p<\infty$.  In particular, the operator is unbounded on $\ell^p(\mathcal L)$
for every finite $p\ge1$.
\end{maintheorem}

Our second discrete example is the family of solvable
Baumslag--Solitar groups
\[
\operatorname{BS}(1,n)
   =\langle a,t\mid tat^{-1}=a^n\rangle,
\qquad n\ge2.
\] 
The letters $a$ and $t$ here are unrelated to the elements denoted by the same letters in $\mathcal L$. We use the {\it digit} generating set
\[
 D_n=\{a^\varepsilon t:0\le\varepsilon<n\}^{\pm1}.
\]
Each  $\mathrm{BS}(1,n)$ is finitely presented, torsion-free, and metabelian.

\begin{maintheorem}[Solvable Baumslag--Solitar groups]
\label{thm:baumslag}
For every \(n\ge2\), equip \(\mathrm{BS}(1,n)\) with counting measure and the word metric associated with \(D_n\).
 Then the centered Hardy--Littlewood maximal operator 
 is not of weak type $(p,p)$ for any $1\le p<\infty$.
In particular, the operator is
unbounded on $\ell^p(\mathrm{BS}(1,n))$ for every finite $p\ge1$.
\end{maintheorem}

To the best of our knowledge, Theorems~A and~B give the first examples
of finitely generated groups, equipped with word metrics and counting
measure, for which the centered Hardy--Littlewood maximal operator
fails to be of weak type $(1,1)$.

\bigskip

Our third example is a continuous one, the Lie group $\mathrm{Sol}$.
Let $\mathrm{Sol}=\mathbb R^3$ with the group
law
\begin{equation}\label{eq:sol-law}
(x,y,z)(a,b,c)
   =\bigl(x+e^za,\;y+e^{-z}b,\;z+c\bigr).
\end{equation}
Equivalently, $\mathrm{Sol}=\mathbb R^2\rtimes_\alpha\mathbb R$,
where $z\in\mathbb R$ acts on $\mathbb R^2$ by
$\alpha_z(a,b)=(e^za,e^{-z}b)$.

Equip $\mathrm{Sol}$ with the Riemannian metric
\begin{equation}\label{eq:sol-metric}
ds^2=e^{-2z}\,dx^2+e^{2z}\,dy^2+dz^2.
\end{equation}
Its Riemannian volume measure is
$
d\mu=dx\,dy\,dz,
$
and this is both  a left and a right Haar measure. The group
$\mathrm{Sol}$ is also metabelian and amenable. The geometric analogy
between $\mathrm{Sol}$ and lamplighter groups is familiar in geometric
group theory; see \cite{EskinFisherWhyte}.

\begin{maintheorem}[The $\mathrm{Sol}$ group]\label{thm:sol}
Let $d$ be the Riemannian distance determined by the metric above,
and let $\mu$ be Riemannian volume. The centered Hardy--Littlewood
maximal operator $M_d$ on $\mathrm{Sol}$ is not of weak type $(p,p)$
for any $1\le p<\infty$. In particular, $M_d$ is unbounded on
$L^p(\mathrm{Sol},\mu)$ for every finite $p\ge1$.
\end{maintheorem}

The same conclusion holds for every left-invariant Riemannian metric
on $\mathrm{Sol}$, with its Riemannian volume;
see Corollary~\ref{cor:sol-all-riemannian}.

\subsection{Common mechanism}
In contrast with the counterexamples mentioned above, all three of our examples are homogeneous.
For the metrics in Theorems~\ref{thm:lamplighter}, \ref{thm:baumslag}, \ref{thm:sol},
all three examples have pure exponential ball growth:
the measures of balls of radius $r$ are comparable to
$2^r$, $n^r$, and $e^r$, respectively,
for all sufficiently large $r$.

For all three examples, the failure of weak type follows from the same elementary criterion.

\begin{lemma}[Mechanism  for failure of weak type]\label{lem:fixed-level}
Let $(X,d,\mu)$ be a metric measure space in which every metric ball
has positive finite measure, and let $M_d$ denote the centered
Hardy--Littlewood maximal operator. Suppose that there exist a
constant $c>0$ and sequences of measurable sets  $H_N,V_N\subseteq X$
with 
$
0<\mu(H_N),\mu(V_N)<\infty
$ such that 
$$
M_d\1_{H_N}\ge c\ \hbox{on }V_N,
\quad
\frac{\mu(V_N)}{\mu(H_N)}\longrightarrow\infty.
$$

Then $M_d$ is not of weak type $(p,p)$ for any $1\le p<\infty$.
Consequently, it is unbounded on $L^p(X,\mu)$ for every finite
$p\ge1$.
\end{lemma}

\begin{proof}
Fix $1\le p<\infty$ and set $\lambda=c/2$. Since
$
M_d{\mathbf 1}_{H_N}\ge c $ on $V_N$, 
we have
$
V_N\subset \{x:M_d{\mathbf 1}_{H_N}(x)>\lambda\}.
$
If the weak type $(p,p)$ estimate held with constant $C_p$, it would follow that
\[
\mu(V_N)
\le C_p\lambda^{-p}\int_X |{\mathbf 1}_{H_N}|^p\,d\mu
= C_p(2/c)^p\mu(H_N).
\]
But
\[
\frac{\mu(V_N)}{\mu(H_N)}\longrightarrow\infty,
\]
so this inequality fails for all sufficiently large $N$. Hence $M_d$ is not
of weak type $(p,p)$. 

Since boundedness on \(L^p(X,\mu)\) would imply weak type \((p,p)\), \(M_d\) is unbounded on \(L^p(X,\mu)\).
\end{proof}

We note that the same sets \(H_N,V_N\) and the same level \(c/2\) work for every finite \(p\).

\bigskip
Each of the three examples carries a natural height function, and the construction is organized by its level sets.
In the applications below, $H_N$ denotes the horizontal target layer
and $V_N$ the vertically extended set of centers. 

The set $H_N$ lies at height $N$, or between heights $N$ and $N+1$.
The set $V_N$ extends from height $0$ to height $N$,
whose horizontal size does not depend on the height.
Thus $\mu(V_N)/\mu(H_N)$ grows linearly with $N$.
For each $g\in V_N$, upward paths reach points of $H_N$
inside a ball $B_d(g,r)$ such that
\[
  \frac{\mu(H_N\cap B_d(g,r))}{\mu(B_d(g,r))}\ge c.
\]
Here $r$ depends on $N$ and the height of $g$,
but $c>0$ is independent of $N$ and $g$.

The pure exponential ball growth is important in the argument.
For the switch–walk–switch metric, a path of \(d\) upward steps passes through \(d+1\) lamp positions. Each lamp can be chosen independently, giving \(2^{d+1}\) possible final configurations. The whole ball has at most \(C2^d\) elements.
For the digit metric on \(BS(1,n)\), at each upward step we choose one of the \(n\) digits \(0,\ldots,n-1\). After \(d\) steps this gives \(n^d\) distinct elements, while the whole ball has at most \(C_n n^d\) elements.
In Sol, moving upward through height \(t\) reaches a terminal box of volume \(4e^t\), while the ball of radius \(t+3\) has volume at most \(Ce^{t+3}\).
Thus in each case, a fixed positive proportion of a suitable ball lies in \(H_N\).

\bigskip

Sections~\ref{sec:sws}, \ref{sec:bs}, and~\ref{sec:sol} prove Theorems A, B, and C, respectively. 
Section~\ref{sec:remarks} discusses consequences of the fixed-level construction and the role of the metric.

\section{The lamplighter group}\label{sec:sws}

\subsection{Switch--walk--switch word metric}\label{sec:group}

Let
$$
\mathcal A
 =\{\eta:\mathbb Z\to\mathbb Z/2\mathbb Z:
          \operatorname{supp}\eta\ \text{is finite}\},
$$
where
$
\operatorname{supp}\eta
 =\{m\in\mathbb Z:\eta(m)\ne0\}
$, 
with pointwise addition. For each $j\in\mathbb Z$, define
$\tau_j:\mathcal A\to\mathcal A$ by
\[
(\tau_j\eta)(m)=\eta(m-j)
\quad (\eta\in\mathcal A,\ m\in\mathbb Z).
\]
Each $\tau_j$ is an automorphism of $\mathcal A$, and
$\tau_j\tau_k=\tau_{j+k}$. Thus the semidirect product
$\mathcal L=\mathcal A\rtimes_\tau\mathbb Z$ has underlying set
$\mathcal A\times\mathbb Z$ and multiplication
$$
(\eta,j)(\xi,k)=(\eta+\tau_j\xi,j+k).
$$
The height homomorphism
$h_{\mathcal L}:\mathcal L\to\mathbb Z$ is given by
$h_{\mathcal L}(\eta,j)=j$.

For $i\in\Z$, let $\delta_i\in\mathcal A$ be the indicator of $\{i\}$,
with values in $\Ftwo$. With $a=(\delta_0,0)$ and $t=(0,1)$, one has
\[
 (\eta,j)a=(\eta+\delta_j,j),\qquad
 (\eta,j)t^{\pm1}=(\eta,j\pm1),\qquad
 t^iat^{-i}=(\delta_i,0).
\]

Every $\eta\in\mathcal A$ is a finite sum of the functions $\delta_i$,
so these identities imply that $a$ and $t$ generate $\LL$.
We call $\eta$ the {\it lamp configuration} and $j$ the {\it cursor position}.
We call the set of elements with cursor position \(j\) the {\it level} \(j\).

Thus right multiplication by \(a\) switches the lamp at the cursor, while right multiplication by \(t\) or \(t^{-1}\) moves the cursor one step to the right or left, respectively. 
We use the switch-walk-switch generating set
\[
S_{\mathrm{sws}}=\{t,\,at,\,ta,\,ata\}^{\pm1}.
\]
This set generates $\mathcal L$ since it contains $t$ and $at$, and
$a=(at)t^{-1}$.
Write $|g|_{\mathrm{sws}}$ for the  least number of these generators
needed to represent $g$, and put
\[
B_{\mathrm{sws}}(r)
   =\{g\in\mathcal L:|g|_{\mathrm{sws}}\le r\},
\quad
r\in\mathbb N_0.
\]

For the left-invariant distance $|x^{-1}y|_{\mathrm{sws}}$, the ball
centered at $x$ is $xB_{\mathrm{sws}}(r)$. Left multiplication is a bijection,
so every such ball has cardinality $|B_{\mathrm{sws}}(r)|$.
Consequently, for every finite set \(F\subset \mathcal L\),
\begin{equation}\label{eq:indicator-average-sws}
 M_{\mathrm{sws}}\one_F(x)
 =\sup_{r\in\Nzero}
   \frac{|xB_{\mathrm{sws}}(r)\cap F|}{|B_{\mathrm{sws}}(r)|}.
\end{equation}

Since the word metric is integer-valued, taking \(r\in\mathbb N_0\) gives the same family of balls as taking positive real radii.

\subsection{Ball growth}

\begin{lemma}[Ball growth bound]\label{lem:sws-growth}
For every integer $r\ge0$,
\begin{equation}\label{eq:sws-ball}
 |B_{\mathrm{sws}}(r)|\le32\,2^r.
\end{equation}
\end{lemma}

\begin{proof}

Let $g=(\eta,j)\in B_{\mathrm{sws}}(r)$. We first consider the case 
that $j \ge 0$. Choose a word of length
at most $r$ representing $g$. Its cursor follows a nearest-neighbor
path from $0$ to $j$. Let $[L,R]$ be the smallest interval containing
this path, and set $w=R-L$. Then, $0,j\in[L,R]$ and $w\ge j$.

Every switch takes place at a visited site, so \(\operatorname{supp}\eta\subseteq[L,R]\). Since the path visits both endpoints,
\begin{equation}\label{eq:sws-interval-cost}
 2w-j\le r.
\end{equation}

Put \(u=w-j\), so that \(w=j+u\). Then \(u\ge 0\), and \eqref{eq:sws-interval-cost} gives
\(j+2u\le r\), so \(0\le u\le (r-j)/2\).

For fixed \(j\) and \(u\), there are \(u+1\) possible intervals of width
\(w=j+u\) containing both \(0\) and \(j\) since their left endpoints range from
\(-u\) to \(0\). Each such interval supports at most \(2^{j+u+1}\) lamp
configurations. This may count an element more than once, which is harmless
for an upper bound. 

Hence, summing over all \(j\) and \(u\) with
\(0\le j\le r\) and \(0\le u\le (r-j)/2\), we obtain
\begin{align*}
 \bigl|\{(\eta,j)\in B_{\mathrm{sws}}(r):j\ge 0\}\bigr|
 &\le\sum_{j=0}^r\sum_{0\le u\le(r-j)/2}(u+1)2^{j+u+1}\\
 &=2\,2^r\sum_{s=0}^r\sum_{0\le u\le s/2}(u+1)2^{-s+u}\\
 &\le2\,2^r\sum_{u=0}^{\infty}(u+1)2^u
                      \sum_{s=2u}^{\infty}2^{-s}\\
 &=4\,2^r\sum_{u=0}^{\infty}(u+1)2^{-u}=16\,2^r.
\end{align*}
In the second line we set $s=r-j$. 
Since all summands are nonnegative, we extend the sum over \(s\)
from \(0 \le s\le r\) to \( 0 \le s<\infty\), then interchange the order of summation.
 In the last
line we used $\sum_{s=2u}^{\infty}2^{-s}=2^{1-2u}$ and
$\sum_{u\ge0}(u+1)2^{-u}=4$.

By reflection, the same bound holds for \(j<0\). Hence \(|B_{\mathrm{sws}}(r)|\le 32\,2^r\).
\end{proof}

\subsection{The sets \(H_N\) and $V_N$}
For an integer $N\ge1$, define
$$
\mathcal K_N
 =\{\eta\in\mathcal A:
\operatorname{supp}\eta\subseteq\{0,1,\ldots,N\}\}.
$$
We have $|\mathcal K_N|=2^{N+1}$.
Define
$$
 H_N=\{(\eta,N):\eta\in\mathcal K_N\},\quad
 V_N=\{(\eta,j):\eta\in\mathcal K_N,\ 0\le j\le N\}.
$$
Thus $V_N$ consists of $N+1$ cursor levels, each containing the same
$2^{N+1}$ configurations, and $H_N$ is its top level. We have 
\begin{equation}\label{eq:EN-XN-sizes}
|H_N|=2^{N+1},
\quad
|V_N|=(N+1)|H_N|.
\end{equation}

\begin{proposition}\label{prop:sws}
For every $N\ge1$ and every $x\in V_N$,
\[
 M_{\mathrm{sws}}\one_{H_N}(x)\ge\frac1{16}.
\]
Hence,
\begin{equation}\label{eq:sws-halo}
 \bigl|\{x:M_{\mathrm{sws}}\one_{H_N}(x)>1/32\}\bigr|
 \ge(N+1)|H_N|.
\end{equation}
\end{proposition}

\begin{proof}
Fix $x=(\eta,j)\in V_N$ and put $d=N-j$. 
If \(d=0\), then \(x\in H_N\), and the radius-zero average is \(1\).
Hence suppose \(d\ge1\).

For each subset \(E\subseteq\{0,1,\ldots,d\}\), let
\(\chi_E\in A\) be its indicator function, viewed as a
\(\mathbb Z/2\mathbb Z\)-valued function, and put
$
h_E=(\chi_E,d).
$
There are \(2^{d+1}\) such subsets \(E\), and the corresponding elements
\(h_E\) are distinct. Hence
\[
\bigl|\{h_E:E\subseteq\{0,1,\ldots,d\}\}\bigr|=2^{d+1}.
\]

Each \(h_E\) is obtained by a monotone walk from \(0\) to \(d\), represented by a word of length \(d\) in \(S_{\rm sws}\): switch at a site if needed before leaving it, and switch at \(d\) if needed after reaching it.

 More explicitly, if
\(\varepsilon_i=\chi_E(i)\), then the word
\[
(a^{\varepsilon_0}t)(a^{\varepsilon_1}t)\cdots
(a^{\varepsilon_{d-2}}t)(a^{\varepsilon_{d-1}}ta^{\varepsilon_d})
\]
has \(d\) elements of \(S_{\mathrm{sws}}\) in the product and represents \(h_E\).
When \(d=1\), it consists only of the last element. Therefore
\(h_E\in B_{\mathrm{sws}}(d)\).
Then
\[
xh_E=(\eta+\tau_j\chi_E,N)\in H_N,
\]
since $\operatorname{supp}(\tau_j\chi_E)\subseteq\{j,\ldots,N\}$.
Thus
\[
\{xh_E:E\subseteq\{0,\ldots,d\}\}
\subseteq (xB_{\mathrm{sws}}(d)\cap H_N).
\]

Since left multiplication by \(x\) is injective,
\[
\begin{aligned}
|xB_{\mathrm{sws}}(d)\cap H_N|
&\ge \bigl|\{xh_E:E\subseteq\{0,\ldots,d\}\}\bigr|\\
&= \bigl|\{h_E:E\subseteq\{0,\ldots,d\}\}\bigr|
=2^{d+1}.
\end{aligned}
\]
Taking $r=d$ in \eqref{eq:indicator-average-sws}, 
Lemma~\ref{lem:sws-growth} gives 
\[
M_{\mathrm{sws}}\mathbf 1_{H_N}(x)
 \ge \frac{2^{d+1}}{|B_{\mathrm{sws}}(d)|}
 \ge \frac1{16}.
\]
Since \(x\in V_N\) was arbitrary,
\[
V_N\subseteq \{x:M_{\mathrm{sws}}\mathbf {\mathbf 1}_{H_N}(x)>1/32\}.
\]
Together with \eqref{eq:EN-XN-sizes}, this gives \eqref{eq:sws-halo}.
\end{proof}

\begin{proof}[Proof of Theorem~\ref{thm:lamplighter}]
By Proposition~\ref{prop:sws},
\[
M_{\mathrm{sws}}\mathbf {\mathbf 1}_{H_N}\ge \frac1{16}
\quad\text{on }V_N,
\]
while \eqref{eq:EN-XN-sizes} gives
\[
\frac{|V_N|}{|H_N|}=N+1\longrightarrow\infty.
\]
Thus Lemma~\ref{lem:fixed-level} applies with \(c=1/16\), and proves the theorem.
\end{proof}

\section{The Baumslag--Solitar groups $\BS(1,n)$}\label{sec:bs}

For  $n\ge2$, write
\[
 G_n=\BS(1,n)=\langle a,t\mid tat^{-1}=a^n\rangle.
\]
We use a generating set adapted to the base-$n$ structure of $G_n$.
This gives a transparent proof, parallel to the
switch--walk--switch argument for the lamplighter group.

\subsection{The digit word metric}

Let 
$\mathbb Z[1/n]\rtimes\mathbb Z
$
be the semidirect product, 
where \(j\in\mathbb Z\) acts on \(\mathbb Z[1/n]\) by multiplication by
\(n^j\). Thus
\begin{equation}\label{eq:bs-digit-law}
(x,j)(y,k)=(x+n^j y,j+k).
\end{equation}
The map
\[
a\longmapsto(1,0),\quad t\longmapsto(0,1)
\]
defines an isomorphism \(G_n\to \mathbb Z[1/n]\rtimes\mathbb Z\).
Indeed, the relation \(tat^{-1}=a^n\) is satisfied, and the map is clearly surjective. 
Injectivity follows from Britton's lemma.
We use this identification and
write elements of \(G_n\) as \((x,j)\).
We call $j$ the {\it height} of $(x,j)$.

Define a symmetric generating set of $G_n$:
$$
 D_n=\{a^\varepsilon t:0\le\varepsilon<n\}^{\pm1}.
$$
 It generates $G_n$
because it contains $t$ and $at$, and $a=(at)t^{-1}$.

For \(0\le \varepsilon<n\), right multiplication by generators gives
\begin{equation}
(x,j)(a^\varepsilon t)=(x+\varepsilon n^j,j+1),
\quad
(x,j)(a^\varepsilon t)^{-1}=(x-\varepsilon n^{j-1},j-1).
\label{eq:bs-digit-steps}
\end{equation}
Thus every generator changes the height by one in absolute value. 
In particular, a word of \(r\) elements chosen from
\(\{a^\varepsilon t:0\le \varepsilon<n\}\) ends at height \(r\)
and is geodesic with respect to $D_n$.
 
 We call \(D_n\) the {\it  digit generating set} and its word metric the {\it digit word metric}.
For the generators \(a^\varepsilon t\), right multiplication adds \(\varepsilon n^j\) to the first coordinate in \eqref{eq:bs-digit-steps}, where \(\varepsilon\in\{0,\ldots,n-1\}\). Thus \(\varepsilon\) plays the role of a base-\(n\) digit. This explains the terminology.

Write $|g|_{D_n}$ for the word length and  define
\[
 B_{D_n}(r)=\{g\in G_n:|g|_{D_n}\le r\},\quad r\in\Nzero.
\]
Let $M_{D_n}$ be the centered maximal operator for this word metric
and counting measure. The ball centered at $g$ is $gB_{D_n}(r)$, and left
multiplication preserves cardinality. Hence, for every finite set
$F\subset G_n$,
\begin{equation}\label{eq:bs-digit-average}
 M_{D_n}\one_F(g)
 =\sup_{r\in\Nzero}\frac{|gB_{D_n}(r)\cap F|}{|B_{D_n}(r)|}.
\end{equation}

\subsection{Ball growth}

\begin{lemma}[Ball growth bound]\label{lem:bs-digit-growth}
There is a constant $C_n>0$ such that, for every integer $r\ge0$,
\begin{equation}\label{eq:bs-digit-ball}
|B_{D_n}(r)|\le C_n n^r.
\end{equation}
\end{lemma}
Height-monotone paths of length at most \(r\) produce \(O(n^r)\) points. The point of the lemma is that paths with backtracking change this only by a constant factor. It is easy to get an \(O(rn^r)\) bound, but that would not be sufficient for our application.

\begin{proof}
Fix $r \ge 0$. 
We first count elements $g=(x,j)\in B_{D_n}(r)$ with $j\ge0$.
For such $g$, choose a shortest word  representing $g$.
Its length is at most $r$. This word gives a path in the second coordinate  that starts at \(0\) and ends at \(j\), changing by one at each step.

Let \(m\) and \(M\) be the minimum and maximum heights (the second coordinates) visited. We define the width by 
\[
w=M-m.
\]
Then \(m\le 0\le j\le M\) and \(0\le w\le r\). If \(w=0\), then
\(g=e\), contributing one element.

Now fix \(w\ge 1\).
Since the path visits both \(m\) and \(M\), its length is at least 
$
2w-j$. Hence $2w-j \le r.
$
Thus $2w-r\le j\le w.$
Also, combined with \(M=m+w\ge j\), we have 
$
w-r\le m\le 0.$

Those bounds on $j$ and $m$ give at most \(r-w+1\) possible values
of \(j\) and at most \(r-w+1\) possible values of \(m\). Hence there
are at most
\begin{equation}\label{eq:bs-range}
    (r-w+1)^2
\end{equation}
possible pairs \((m,j)\)
for the fixed $r$ and $w$.

\bigskip
For each such pair \((m,j)\), we now count the possible first coordinates
\(x\). By \eqref{eq:bs-digit-steps}, a crossing from height \(h\) to
\(h+1\) contributes \(\varepsilon n^h\) to \(x\), and a crossing from
\(h+1\) to \(h\) contributes \(-\varepsilon n^h\), where
\(0\le\varepsilon<n\).

The path crosses between heights \(h\) and \(h+1\) at least once for
each \(m\le h<M\). Choose one such crossing for each \(h\). There are
\(w\) such crossings, so there are at most \(r-w\) further crossings.

Since all crossings occur between heights \(m\) and \(M\), we have
\(x\in n^m\mathbb Z\). After division by \(n^m\), the chosen crossings
contribute at most
\[
(n-1)\sum_{i=0}^{w-1}n^i=n^w-1.
\]
Each remaining  crossing contributes at most
\((n-1)n^{w-1}\) since 
$0 \le \varepsilon \le n-1$. 
Hence the total contribution of the remaining  ones is at most
$(r-w)(n-1)n^{w-1}$.
Consequently, we have 
$$
\frac{|x|}{n^m}
\le
n^w-1+(r-w)(n-1)n^{w-1} \le (r-w+1)n^w.
$$
Since \(x/n^m\) is an integer, the number of possible values of \(x\)
is at most
$
2(r-w+1)n^w.
$

Combined with \eqref{eq:bs-range}, for the fixed \(w\ge1\),  there are at most
$
2(r-w+1)^3n^w
$
possible elements $(x,j)$ with \(j\ge0\).

\bigskip
Finally, summing over $w$, and adding the identity element, we obtain
\[
\bigl|\{(x,j)\in B_{D_n}(r):j\ge0\}\bigr|
\le
2\sum_{w=0}^r (r-w+1)^3n^w.
\]
Setting \(k=r-w\), we have 
\[
2\sum_{w=0}^r (r-w+1)^3n^w
=
2n^r\sum_{k=0}^r (k+1)^3n^{-k}
\le
2n^r\sum_{k=0}^{\infty}(k+1)^3n^{-k}.
\]

For \(j<0\), replace \(g\) by \(g^{-1}\). Inversion preserves word length and changes the sign of the height. 
We have the same bound as $j\ge0$.

Thus \eqref{eq:bs-digit-ball} holds with
\[
C_n
=
4\sum_{k=0}^{\infty}(k+1)^3n^{-k}
<\infty.
\]
\end{proof}

\subsection{The sets $H_N$ and $V_N$}

For an integer $N\ge1$, define
$$
\begin{aligned}
 H_N&=\{(x,N):x\in\Z,\ 0\le x<2n^N\},\\
 V_N&=\{(x,j):x,j\in\Z,\ 0\le x<n^N,\ 0\le j\le N\}.
\end{aligned}
$$
The set $H_N$ lies at height $N$. The set $V_N$ has $N+1$ levels,
each with $n^N$ elements. Therefore
\begin{equation}\label{eq:bs-digit-sizes}
 |H_N|=2n^N,\quad |V_N|=(N+1)n^N,
 \quad \frac{|V_N|}{|H_N|}=\frac{N+1}{2}.
\end{equation}

\begin{proposition}\label{prop:bs-digit}
For every $N\ge1$ and every $g\in V_N$,
\[
                         M_{D_n}\one_{H_N}(g)\ge\frac1{C_n}.
\]
Hence,
\begin{equation}\label{eq:bs-digit-halo}
 \bigl|\{g\in G_n:M_{D_n}\one_{H_N}(g)>1/(2C_n)\}\bigr|
                   \ge\frac{N+1}{2}|H_N|.
\end{equation}
\end{proposition}
\begin{proof}
Fix $g=(x,j)\in V_N$ and put $r=N-j$.
For each choice

$\varepsilon_0,\ldots,\varepsilon_{r-1}\in\{0,1,\ldots,n-1\}$,
the group law \eqref{eq:bs-digit-law} gives
\begin{equation}\label{eq:bs-digit-words}
 (a^{\varepsilon_0}t)\cdots(a^{\varepsilon_{r-1}}t)
      =\left(\sum_{i=0}^{r-1}\varepsilon_i n^i,r\right).
\end{equation}
As \(\varepsilon_0,\ldots,\varepsilon_{r-1}\) vary, the sum in \eqref{eq:bs-digit-words} runs through each integer \(k\) with \(0\le k<n^r\) exactly once.
 Each product uses $r$ elements of $D_n$. Hence the set
\[
Q_r=\{(k,r):k\in\Z,\ 0\le k<n^r\}
\]
satisfies
\begin{equation}\label{eq:bs-digit-Qr}
|Q_r|=n^r,\quad  Q_r\subset B_{D_n}(r).
\end{equation}
For $r=0$, the product is empty and $Q_0=\{(0,0)\}$, so these
statements remain valid.

For every $(k,r)\in Q_r$, left multiplication by $g=(x,j)$ gives
\[
g(k,r)=(x+n^j k,N).
\]
The first coordinate is an integer, because $x,k\in\Z$ and $j\ge0$,
and it satisfies
\[
0\le x+n^j k<n^N+n^j n^r=2n^N.
\]
Thus $gQ_r\subseteq H_N$. Together with \eqref{eq:bs-digit-Qr}, this
proves
\[
gQ_r\subset (gB_{D_n}(r)\cap H_N).
\]

Left multiplication by $g$ is injective, so
\[
|gB_{D_n}(r)\cap H_N| \ge|gQ_r|=|Q_r|=n^r.
\]
Taking this radius in \eqref{eq:bs-digit-average},
Lemma~\ref{lem:bs-digit-growth} gives

\[
 M_{D_n}\one_{H_N}(g)
 \ge\frac{|gB_{D_n}(r)\cap H_N|}{|B_{D_n}(r)|}
 \ge\frac{n^r}{C_n n^r}=\frac1{C_n}.
\]
Since \(g\in V_N\) was arbitrary,
\[
 V_N\subseteq\{g\in G_n:M_{D_n}\one_{H_N}(g)>1/(2C_n)\}.
\]
Together with \eqref{eq:bs-digit-sizes}, this gives
\eqref{eq:bs-digit-halo}.
\end{proof}

\begin{proof}[Proof of Theorem~\ref{thm:baumslag}]

By Proposition~\ref{prop:bs-digit},
\[
M_{D_n}\one_{H_N}\ge\frac1{C_n} \quad\text{on }V_N,
\]
while \eqref{eq:bs-digit-sizes} gives
\[
\frac{|V_N|}{|H_N|}  
=\frac{N+1}{2}\longrightarrow\infty.
\]
Thus Lemma~\ref{lem:fixed-level} applies with \(c=1/C_n\), and proves the theorem.
\end{proof}

\section{The Sol group}\label{sec:sol}

\subsection{Geometry of Sol}
Let $\Sol$ be the Lie group $\R^3$ with multiplication \eqref{eq:sol-law}. Its
identity and inverse are
\[
 o=(0,0,0),\qquad (x,y,z)^{-1}=(-e^{-z}x,-e^zy,-z).
\]
The map $h_{\Sol}:\Sol\to\R$, $h_{\Sol}(x,y,z)=z$, is a group
homomorphism, called the {\it height}. Its kernel is the horizontal subgroup
$\R^2\times\{0\}$. Conjugation by $(0,0,z)$
acts on the horizontal subgroup as
\[
 (0,0,z)(a,b,0)(0,0,-z)=(e^za,e^{-z}b,0).
\]

For the metric \eqref{eq:sol-metric}, the length of a piecewise smooth
curve $\gamma(s)=(x(s),y(s),z(s))$ is
\begin{equation}\label{eq:sol-length}
 L(\gamma)=\int
 \sqrt{e^{-2z(s)}x'(s)^2+e^{2z(s)}y'(s)^2+z'(s)^2}\,ds.
\end{equation}
In particular,
a vertical segment has length equal to its change in height. At fixed
height $z$, a segment changing only $x$ by $a$ has length $e^{-z}|a|$,
and one changing only $y$ by $b$ has length $e^z|b|$.

Left multiplication by $q=(x_0,y_0,z_0)$ is the map
\begin{equation}\label{eq:sol-left-translation}
 L_q(a,b,c)=(x_0+e^{z_0}a,\ y_0+e^{-z_0}b,\ z_0+c).
\end{equation}
Pulling back \eqref{eq:sol-metric} under this map gives
\[
 e^{-2(z_0+c)}e^{2z_0}\,da^2
 +e^{2(z_0+c)}e^{-2z_0}\,db^2+dc^2
 =e^{-2c}\,da^2+e^{2c}\,db^2+dc^2.
\]
Hence $L_q$ is an isometry, and $B_d(q,r)=qB_d(o,r)$.

The metric tensor has determinant $e^{-2z}e^{2z}=1$, so its volume
element is $d\mu=dx\,dy\,dz$.
Since left multiplication is an isometry, it preserves Riemannian
volume. Hence
\begin{equation}
\mu(B_d(q,r))=\mu(B_d(o,r)).
\label{eq:sol-homogeneity}
\end{equation}

\subsection{Sphere and ball growth}\label{sec:sol-growth}
We use the following theorem of Schwartz on the area growth of metric spheres in \(\mathrm{Sol}\).
\begin{theorem}[Sphere growth bound {\cite[Theorem 1.1]{Schwartz}}]\label{th-sch}
For the metric \eqref{eq:sol-metric}, there exists \(r_0>0\) such that
the Riemannian area \(A(r)\) of the metric sphere of radius \(r\) satisfies
\[
A(r)\le 20\pi e^r
\qquad (r\ge r_0).
\]
\end{theorem}

We now deduce the corresponding volume estimate.
\begin{lemma}[Ball growth bound]\label{lem:sol-growth}
There is a constant $C_{\Sol}\ge1$ such that
$$
 \mu(B_d(o,r))\le C_{\Sol}e^r\qquad(r\ge0).
$$
\end{lemma}

\begin{proof}
Let
$
V(r)=\mu(B_d(o,r)).
$
By the coarea formula (see \cite[Theorem~3.2.22]{Fed}) and Theorem \ref{th-sch}, 
\[
V(r)-V(r_0)
 =\int_{r_0}^r A(s)\,ds
 \le 20\pi\int_{r_0}^r e^s\,ds
 \le 20\pi e^r
 \qquad (r\ge r_0).
\]
After increasing the constant to cover \(0\le r\le r_0\), we obtain
\[
V(r)\le C_{\mathrm{Sol}}e^r
\qquad (r\ge0).
\]
\end{proof}

\subsection{The sets $H_N$ and $V_N$}\label{sec:sol-counterexample}

\begin{lemma}\label{lem:sol-terminal-box}
For $t\ge0$, let
\[
 Q_t=[-e^t,e^t]\times[-1,1]\times[t,t+1].
\]
Then $Q_t\subset B_d(o,t+3)$ and $\mu(Q_t)=4e^t$.
\end{lemma}

\begin{proof}
Given $(a,b,c)\in Q_t$, join $o$ to it by three segments:
\[
 (0,0,0)\longrightarrow(0,b,0)
 \longrightarrow(0,b,c)\longrightarrow(a,b,c).
\]
Their lengths, computed using \eqref{eq:sol-length}, are
$|b|$, $c$, and $e^{-c}|a|$, respectively. Since $c\in[t,t+1]$,
the total length is at most
\[
 1+(t+1)+e^{-c}e^t\le t+3.
\]
Hence $d(o,(a,b,c))\le t+3$, so 
$Q_t \subset B_d(o,t+3)$.

The three side lengths of \(Q_t\) are \(2e^t\), \(2\), and \(1\). Since \(d\mu=dx\,dy\,dz\), we have
$\mu(Q_t)=4e^t$.
\end{proof}

For $N\ge1$, define the compact sets
\begin{align*}
 H_N&=[-2e^N,2e^N]\times[-2,2]\times[N,N+1],
       \label{eq:sol-FN}\\
 V_N&=[-e^N,e^N]\times[-1,1]\times[0,N].
\end{align*}
Their volumes are
\begin{equation}\label{eq:sol-sizes}
 \mu(H_N)=16e^N,\quad
 \mu(V_N)=4Ne^N,\quad
 \frac{\mu(V_N)}{\mu(H_N)}=\frac N4.
\end{equation}

\begin{proposition}\label{prop:sol-fixed-level}
Set $c_{\Sol}=4/(C_{\Sol}e^3)>0$. For every $N\ge1$,
$$
 M_d\one_{H_N}(q)\ge c_{\Sol}\qquad(q\in V_N).
$$
\end{proposition}

\begin{proof}
Fix \(q=(x_0,y_0,z_0)\in V_N\) and set
$
t=N-z_0\ge 0.
$

We first check that \(qQ_t\subset H_N\). For \((a,b,c)\in Q_t\),
formula \eqref{eq:sol-left-translation} and the definitions of \(V_N\)
and \(Q_t\) give
\[
|x_0+e^{z_0}a|
 \le e^N+e^{z_0}e^t
 =2e^N,
\]
\[
|y_0+e^{-z_0}b|
 \le 1+e^{-z_0}
 \le 2,
\]
\[
z_0+c\in[z_0+t,z_0+t+1]=[N,N+1].
\]
Thus
$
qQ_t\subset H_N.
$

By Lemma~\ref{lem:sol-terminal-box} and left invariance of the metric,
$
qQ_t\subset B_d(q,t+3).
$
Hence
\[
qQ_t\subset (H_N\cap B_d(q,t+3)).
\]
Moreover, left invariance of volume gives
\[
\mu(qQ_t)=\mu(Q_t)=4e^t.
\]

By \eqref{eq:sol-homogeneity} and Lemma~\ref{lem:sol-growth},
\[
\mu(B_d(q,t+3))
 =\mu(B_d(o,t+3))
 \le C_{\mathrm{Sol}}e^{t+3}.
\]
Therefore
\[
M_d{\mathbf 1}_{H_N}(q)
 \ge
 \frac{\mu(H_N\cap B_d(q,t+3))}
      {\mu(B_d(q,t+3))}
 \ge
 \frac{4e^t}{C_{\mathrm{Sol}}e^{t+3}}
 =
 \frac{4}{C_{\mathrm{Sol}}e^3}
 =
 c_{\mathrm{Sol}}.
\]
\end{proof}

\begin{proof}[Proof of Theorem~\ref{thm:sol}]
By Proposition~\ref{prop:sol-fixed-level},
\[
M_d{\mathbf 1}_{H_N}\ge c_{\mathrm{Sol}}
\quad\text{on }V_N,
\]
while \eqref{eq:sol-sizes} gives
\[
\frac{\mu(V_N)}{\mu(H_N)}
=\frac{N}{4}\longrightarrow\infty.
\]
Thus Lemma~\ref{lem:fixed-level} applies with \(c=c_{\mathrm{Sol}}\), and proves the theorem.
\end{proof}

\begin{corollary}[Sol with general left-invariant Riemannian metric]\label{cor:sol-all-riemannian}
For every left-invariant Riemannian metric on $\mathrm{Sol}$,
with its Riemannian volume measure, the centered Hardy--Littlewood
maximal operator is not of weak type $(p,p)$ for any
$1\le p<\infty$.
In particular, it is unbounded on $L^p$ for every finite $p\ge1$.
\end{corollary}

\begin{proof}
Let $\widetilde d$ be a left-invariant Riemannian distance on
$\mathrm{Sol}$, and let $d$ be the standard distance used above.
By \cite[Theorem B]{LeDonne}, after rescaling
$\widetilde d$, there is a constant $A\ge0$ such that
\[
  |\widetilde d(x,y)-d(x,y)|\le A
  \quad (x,y\in\mathrm{Sol}).
\]

Rescaling the distance does not change the maximal operator.
The Riemannian volume measures are constant multiples of one
another, since they are both left Haar measures, which are unique up to multiplicative factors. So we use the same measure $\mu$.

Fix $q=(x,y,z)\in V_N$ and put $t=N-z$.
The construction in Proposition~\ref{prop:sol-fixed-level} gives
\[
  qQ_t\subset H_N\cap B_d(q,t+3),
  \quad \mu(qQ_t)=4e^t.
\]
Since
\[
  B_d(q,t+3)
  \subset B_{\widetilde d}(q,t+3+A)
  \subset B_d(q,t+3+2A),
\]
Lemma~\ref{lem:sol-growth} gives
\[
  M_{\widetilde d}\mathbf{1}_{H_N}(q)
  \ge
  \frac{4e^t}{C_{\mathrm{Sol}}e^{t+3+2A}}
  =
  \frac{4}{C_{\mathrm{Sol}}e^{3+2A}}.
\]
This lower bound is independent of $N$ and $q$.
Since $\mu(V_N)/\mu(H_N)=N/4\to\infty$,
Lemma~\ref{lem:fixed-level} proves the result.
\end{proof}

\section{Further remarks}\label{sec:remarks}

\subsection{Failure of endpoint estimates}
For each metric in Theorems~\ref{thm:lamplighter}, \ref{thm:baumslag} and \ref{thm:sol}, the
proof gives sets $H_N,V_N$ and a constant $c>0$ such that
\[
 M_d\one_{H_N}\ge c\quad\hbox{on }V_N,
 \quad \frac{\mu(V_N)}{\mu(H_N)}\longrightarrow\infty.
\]

Since $c$ does not depend on $N$,  for any fixed nondecreasing finite-valued function  $\Phi:[0,\infty)\to[0,\infty)$ 
with $\Phi(0)=0$, 
there is no constant \(C<\infty\) such that
\begin{equation}\label{eq:orlicz}
 \mu\{M_df>\lambda\}
 \le C\int_X\Phi\bigl(|f|/\lambda\bigr)\,d\mu
\end{equation}
for all \(f\) and \(\lambda>0\).
Indeed, set $f=\one_{H_N}$ and
$\lambda=c/2$. The left-hand side is at least $\mu(V_N)$, while the
right-hand side is $C\Phi(2/c)\mu(H_N)$. 
Hence, for sufficiently large $N$, \eqref{eq:orlicz} fails with this choice of \(f\) and \(\lambda\).

Now, take 
\[
 \Phi_\beta(t)=t(1+\log^+t)^\beta\quad(t>0),\quad
 \Phi_\beta(0)=0,
\]
for any fixed $\beta\ge0$, where
$\log^+t=\max\{\log t,0\}$.
This choice implies that the \(L(\log L)^\beta\) endpoint estimate for \(M_d\), namely \(\eqref{eq:orlicz}\) with \(\Phi=\Phi_\beta\), fails for every \(\beta\ge0\).

\subsection{Role of the metric}
Theorems~\ref{thm:lamplighter} and~\ref{thm:baumslag} use the particular
generating sets $S_{\mathrm{sws}}$ and $D_n$. Changing a finite generating set
changes the word metric up to a multiplicative constant. Thus a ball of
radius $r$ for one word metric is contained in a ball of radius $Cr$ for
another. For groups of exponential growth, the latter ball can be much
larger. Therefore the estimates in our proofs do not automatically pass to
another generating set. It remains open whether the same failure holds for
every finite generating set of $\mathcal L$ or $\mathrm{BS}(1,n)$.

By contrast, for $\mathrm{Sol}$, Corollary~\ref{cor:sol-all-riemannian} shows that the
same failure holds for every left-invariant Riemannian metric.

Finally, every centered ball is also allowed in the uncentered maximal
operator. Hence the uncentered maximal operator is at least as large as the
centered one, and has the same failures for all the metrics considered in
this paper.

\bigskip

\medskip
\noindent{\bf Acknowledgements.}
The author  thanks the Mathematical Institute, University of Oxford, for its hospitality during the preparation of this work.

 The author used ChatGPT (GPT-5.6 Sol) to help organize and carry out the estimates in the proof of Lemma~\ref{lem:bs-digit-growth}. All mathematical content is the responsibility of the author.

\bigskip

\noindent
\textsc{Koji Fujiwara}\\
Okinawa Institute of Science and Technology Graduate University,
Okinawa 904-0495, Japan\\
KUIAS/RIMS, Kyoto University, Kyoto, 606-8305, Japan\\
\texttt{koji.fujiwara@oist.jp; kfujiwara@math.kyoto-u.ac.jp}

\end{document}